\documentclass{siamart0216}

\usepackage{amssymb}
\usepackage{enumerate} 
\newtheorem{theo}{Theorem}
\newtheorem{prop}[theo]{Proposition}
\newtheorem{lem}[theo]{Lemma}

\newtheorem{construct}[theo]{Construction}

\newtheorem{rem}[theo]{Remark}

\newcommand{\sm}{\setminus}
\newcommand{\eps}{\varepsilon}
\newcommand{\T}{{\mathcal T}}

\newcommand{\R}{\mathbb{R}}

\newcommand{\C}{{\mathcal C}}
\newcommand{\K}{{\mathcal K}}

\newcommand{\D}{{\mathcal D}}
\newcommand{\U}{{\mathcal U}}

\newcommand{\vb}{{\bf v}}

\newcommand{\xb}{{\bf x}}
\newcommand{\yb}{{\bf y}}

\newcommand{\1}{{\bf 1}}

\newcommand{\chicr}{\chi_{\rm cr}}

\if10     
\usepackage[mathlines]{lineno}
\newcommand*\patchAmsMathEnvironmentForLineno[1]{%
  \expandafter\let\csname old#1\expandafter\endcsname\csname #1\endcsname
  \expandafter\let\csname oldend#1\expandafter\endcsname\csname end#1\endcsname
  \renewenvironment{#1}%
     {\linenomath\csname old#1\endcsname}%
     {\csname oldend#1\endcsname\endlinenomath}}%
\newcommand*\patchBothAmsMathEnvironmentsForLineno[1]{%
 \patchAmsMathEnvironmentForLineno{#1}%
  \patchAmsMathEnvironmentForLineno{#1*}}%
\AtBeginDocument{%
\patchBothAmsMathEnvironmentsForLineno{equation}%
\patchBothAmsMathEnvironmentsForLineno{align}%
\patchBothAmsMathEnvironmentsForLineno{flalign}%
\patchBothAmsMathEnvironmentsForLineno{alignat}%
\patchBothAmsMathEnvironmentsForLineno{gather}%
\patchBothAmsMathEnvironmentsForLineno{multline}%
}
\linenumbers
\fi

\title{An asymptotic multipartite K\"uhn-Osthus theorem}%

\author{Ryan R. Martin\thanks{Iowa State University, 396 Carver Hall, 411 Morrill Road, Ames, Iowa, USA (\email{rymartin@iastate.edu}). This author's research was partially supported by National Science Foundation Grant DMS-0901008, National Security Agency grant H98230-13-1-0226 and Simons Foundation Grant \#353292. The research was partially done while the author was a long-term visitor at the Institute for Mathematics and its Applications (IMA) and would not have been possible without their generous support}%
\and
	Richard Mycroft\thanks{University of Birmingham, Edgbaston, Birmingham B15 2TT, England, UK (\email{r.mycroft@bham.ac.uk}). This author's research was partially supported by EPSRC grant EP/M011771/1}%
\and
	Jozef Skokan\thanks{London School of Economics and Political Science, Houghton Street, London WC2A 2AE, England, UK and Department of Mathematics, University of Illinois, 1409 W. Green Street, Urbana, Illinois, USA (\email{jozef@member.ams.org}) This author's research was partially supported by National Science Foundation Grant DMS-1500121}}

\headers{An asymptotic multipartite K\"uhn-Osthus theorem}{Ryan R. Martin, Richard Mycroft, and Jozef Skokan}

\begin{document}

\maketitle

\begin{abstract}
In this paper we prove an asymptotic multipartite version of a well-known theorem of K\"uhn and Osthus by establishing, for any graph $H$ with chromatic number $r$, the asymptotic multipartite minimum degree threshold which ensures that a large $r$-partite graph $G$ admits a perfect $H$-tiling. We also give the threshold for an $H$-tiling covering all but a linear number of vertices of $G$, in a multipartite analogue of results of Koml\'os and of Shokoufandeh and Zhao.
\end{abstract}


\thispagestyle{empty}
\section{Introduction} \label{sec:intro}

\subsection{Motivation}

Given graphs $G$ and $H$, a simple and natural question to ask is whether it is possible to perfectly tile $G$ with copies of~$H$, that is, to find vertex-disjoint copies of~$H$ in~$G$ which together cover every vertex of $G$. An obvious necessary condition for this is that $|V(H)|$ divides $|V(G)|$, which we assume implicitly throughout this discussion. In the case where $H$ consists of a single edge, a perfect $H$-tiling is simply a perfect matching; a classical theorem of Tutte~\cite{Tu47} gives a characterisation of all graphs for which this is possible, and Edmond's algorithm~\cite{E65} returns a perfect matching (or reports that no such matching exists) in polynomial time. However, if the graph $H$ has a~connected component with at least three vertices, then we see sharply different behaviour. In particular Hell and Kirkpatrick~\cite{KH} showed that, for any fixed graph $H$ of this form, the problem of determining whether a graph $G$ admits a perfect $H$-tiling is NP-hard, so it is unlikely that there exists a `nice' characterisation of such graphs analogous to Tutte's theorem.

Due to this, there has been much study of sufficient conditions which, for a fixed graph $H$, ensure the existence of a perfect $H$-tiling in a~graph~$G$ on~$n$ vertices (we refer the reader to the survey of K\"uhn and Osthus~\cite{KO} for a more detailed overview). The most natural of these are minimum degree conditions; to discuss these we define $\delta(H, n)$ to be the smallest integer $m$ such that any graph $G$ on $n$ vertices with $\delta(G) \geq m$ admits a perfect $H$-tiling. One early sufficient condition was given by the celebrated Hajn\'al-Szemer\'edi theorem~\cite{HSz}, which states that for any integer $r$ we have $\delta(K_r, n) = \tfrac{r-1}{r}n$ (the case $r=3$ was previously given by Corr\'adi and Hajnal~\cite{CH}). Turning to general graphs~$H$, Alon and Yuster~\cite{AY} later showed that $\delta(H, n) \leq \tfrac{\chi(H)-1}{\chi(H)}n + o(n)$; using the Blow-up Lemma Koml\'os, S\'ark\"ozy and Szemer\'edi  \cite{KSSz} then improved this result by replacing the $o(n)$ error term with an additive constant (which cannot be removed in general). In the other direction, Koml\'os \cite{K} introduced the \emph{critical chromatic number} $\chicr(H)$ of~$H$, and observed that for any~$H$ we have $\delta(H, n) \geq \tfrac{\chicr(H)-1}{\chicr(H)}n$. Finally K\"uhn and Osthus~\cite{KO2} completed our understanding by classifying graphs $H$ according to their \emph{greatest common divisor}, and showing that for any~$H$ we have either $\delta(H, n) = \tfrac{\chicr(H)-1}{\chicr(H)}n + O(1)$ or $\delta(H, n) = \tfrac{\chi(H)-1}{\chi(H)}n + O(1)$, according to the value of this parameter.

In parallel with the results described above, much attention has been devoted to the problem of perfectly tiling multipartite graphs; these have presented significant additional challenges. There is a natural multipartite notion of minimum degree: for a $r$-partite graph $G$ with vertex classes $V_1, \dots, V_r$ we define $\delta^*(G)$ to be the largest integer $s$ such that, for any $i \neq j$, any vertex of $V_i$ has at least $s$ neighbours in $V_j$. Similarly as before, let $\delta^*(H, n)$ denote the smallest integer $m$ such that any $\chi(H)$-partite graph $G$ whose $\chi(H)$ vertex classes each have size $n$ and which satisfies $\delta^*(G) \geq m$ admits a perfect $H$-tiling. Fischer~\cite{F} conjectured the natural multipartite analogue of the Hajnal-Szemer\'edi theorem, namely that $\delta^*(K_r, n) = \tfrac{r-1}{r}n$. Perhaps surprisingly, Catlin gave counterexamples demonstrating this natural conjecture to be false for each odd $r \geq 3$. However, for large $n$, Fischer's conjecture is `almost-true', in that Catlin's counterexamples are the only counterexamples to the conjecture, as shown by Keevash and Mycroft~\cite{KM} (this was previously demonstrated for $r=3$ and $r=4$ by Magyar and Martin~\cite{MM} and Martin and Szemer\'edi~\cite{MSz} respectively, whilst an asymptotic form for all $r$ was independently given by Lo and Markstr\"om~\cite{LM}). Subsequently Martin and Skokan~\cite{MS} continued this direction of research by establishing a multipartite analogue of the Alon-Yuster theorem, namely that for any $H$ we have $\delta^*(H, n) \leq \tfrac{\chi(H)-1}{\chi(H)}n + o(n)$. 

\begin{theo}[\cite{MS}]\label{theo:alonyuster}
Let $H$ be a graph on $h$ vertices with $\chi(H)=r\geq 3$. For any $\alpha > 0$ there exists $n_0 = n_0(\alpha,H)$ such that if $G$ is a balanced $r$-partite graph on $rn$ vertices with $\delta^*(G) \geq \left(\frac{r - 1}{r} + \alpha \right) n$, where $n \geq n_0$ is divisible by $h$, then $G$ admits a perfect $H$-tiling.
\end{theo}

In this paper we prove an asymptotic multipartite analogue of the K\"uhn-Osthus theorem (Theorem~\ref{theo:degalpha}), which establishes the asymptotic value of $\delta^*(H, n)$ for any graph $H$ with $\chi(H) \geq 3$. Together with a~theorem of Bush and Zhao~\cite{BZ}, who previously gave the corresponding result for bipartite graphs $H$ up to an additive constant, this determines the asymptotic value of $\delta^*(H, n)$ for every graph $H$.

It is also natural to ask for the minimum degree condition needed to ensure that we can find an $H$-tiling in $G$ covering almost all the vertices of $G$. In the non-partite setting Koml\'os~\cite{K} showed that $\delta(G) \geq \tfrac{\chicr(H)-1}{\chicr(H)}n$ is sufficient to ensure an $H$-tiling covering all but $o(n)$ vertices. He conjectured that in fact this condition guarantees an $H$-tiling covering all but a constant number of vertices, and this was subsequently confirmed by Shokoufandeh and Zhao~\cite{SZ}. Our second main result (Theorem~\ref{theo:almosttiling}) gives a multipartite analogue of Koml\'os's result, namely that any $\chi(H)$-partite graph $G$ whose vertex classes each have size $n$ and which satisfies $\delta^*(G) \geq \tfrac{\chicr(H)-1}{\chicr(H)}n$ admits an $H$-tiling covering all but $o(n)$ vertices of $G$. Again, an analogous result for bipartite graphs $H$ was previously given by Bush and Zhao~\cite{BZ}.

\subsection{Main results}

Let $G$ and $H$ be graphs. An \emph{$H$-tiling} in $G$ is a collection of vertex-disjoint copies of $H$ in $G$; it is \emph{perfect} if every vertex of $G$ is covered by some member of the tiling. Let $H$ be a graph with chromatic number $\chi(H) = r$, and let $\C$ denote the set of proper $r$-colourings of $H$. Then for any proper $r$-coloring $\phi \in \C$ with colour classes $X_1^\phi, \dots, X_r^\phi$, we define
\begin{equation*}
   \D(\phi) := \big\{|X_i^\phi| - |X^\phi_j| : i, j \in [r]\big\} \mbox{ and } \D(H) := \bigcup_{\phi \in \C} \D(\phi) .
\end{equation*}
(Throughout this paper we write $[r]$ to denote the set $\{1, \dots, r\}$.)
The \emph{greatest common divisor} of $H$, denoted $\gcd(H)$, is then defined to be the highest common factor of the set $\D(H)$ if $\D(H)\neq 0$. If $\D(H)=0$ (that is, if every $r$-coloring of $H$ is \emph{equitable}, meaning that all colour classes have the same size) then we write $\gcd(H) = \infty$. We also define
\begin{equation} \label{eq:defsigma}
   \sigma(H) := \min_{\phi \in \C, i \in [r]} \frac{|X_i^\phi|}{|V(H)|}.
\end{equation}
So $0 < \sigma(H) \leq 1/r$, with equality if and only if every $r$-colouring of $H$ is equitable.
The \emph{critical chromatic number} of $H$, introduced by Koml\'os~\cite{K}, is denoted $\chicr(H)$ and is defined by
\begin{equation} \label{eq:defchicr}
  \chicr(H) := \frac{\chi(H)-1}{1-\sigma(H)}.
\end{equation}
So for any graph $H$ we have $\chi(H)-1 < \chicr(H) \leq \chi(H)$, again with equality if and only if every $\chi(H)$-coloring of $H$ is equitable. Note that the definition of $\sigma(H)$ that we use differs by a factor of $|V(H)|$ from that used by K\"uhn and Osthus~\cite{KO2}, but our definition of $\chicr(H)$ is the same.
Finally, following K\"uhn and Osthus~\cite{KO2}, we define
\begin{align*}
   \chi^*(H) := \begin{cases} \chicr(H) & \mbox{if $\gcd(H) = 1$,} \\
	                            \chi(H) & \mbox{otherwise.} \end{cases}
\end{align*}

Recall that if $G$ is an $r$-partite graph with vertex classes $V_1, \dots, V_r$, then the multipartite minimum degree of $G$, denoted $\delta^*(G)$, is defined to be the largest integer $m$ such that for any $i \neq j$ every vertex of $V_i$ has at least $m$ neighbours in $V_j$. Also, we say that $G$ is \emph{balanced} if every vertex class has the same size.

Our first main result is Theorem~\ref{theo:degalpha} in which the optimal degree condition is relaxed by an additive factor which is linear in the number of vertices. This is an asymptotic multipartite version of a theorem of K\"uhn and Osthus~\cite{KO2}.

\begin{theo}\label{theo:degalpha}
Let $H$ be a graph on $h$ vertices with $\chi(H)=r\geq 3$. For any $\alpha > 0$ there exists $n_0=n_0(\alpha,H)$ such that if $G$ is a balanced $r$-partite graph $G$ on $rn$ vertices with $\delta^*(G) \geq \left(1-\frac{1}{\chi^*(H)} + \alpha \right) n$, where $n \geq n_0$ is divisible by $h$, then $G$ contains a perfect $H$-tiling.
\end{theo}

\begin{rem}
In the case where $\gcd(H) = 1$, the proof of Theorem~\ref{theo:degalpha} only uses the weaker assumption that $h$ divides $rn$ (rather than $h$ divides $n$). However, as observed in~\cite{MS}, in the case $\gcd(H) > 1$ we do indeed require that $h$ divides $n$.
\end{rem}

Constructions given in Section~\ref{sec:construct} demonstrate that, for any graph~$H$, Theorem~\ref{theo:degalpha} is best-possible up to the $\alpha n$ error term in the degree condition.
A similar but slightly-different result holds in the case $r=2$; this case was fully settled by Bush and Zhao~\cite{BZ} up to an additive constant. 
\begin{theo}[Bush-Zhao~\cite{BZ}] \label{BZ:degalpha}
For any bipartite graph $H$ there exist constants $n_0=n_0(H)$ and $c=c(H)$ such that if $G$ is a balanced bipartite graph with $n\geq n_0$ vertices in each part, and $|H|$ divides $2n$, then the following statements hold (where $\gcd(H)$ is defined as above and ${\rm gcd}_{\mathrm{cc}}(H)$ is the greatest common divisor of the sizes of the connected components of $H$).
\begin{enumerate}[(a)]
   \item If $\gcd(H)>1$ and ${\rm gcd}_{\mathrm{cc}}(H)=1$, then $\delta(G)\geq (1-1/\chi(H))n+c$  suffices to ensure a perfect $H$-tiling in $G$.
   \item If $\gcd(H)=1$ or ${\rm gcd}_{\mathrm{cc}}(H)>1$, then $\delta(G)\geq (1-1/\chicr(H))n+c$ suffices to ensure a perfect $H$-tiling in $G$.
\end{enumerate}
\end{theo}

The necessity of considering ${\rm gcd}_{\mathrm{cc}}(H)$ is unique to the case of $r=2$, which we discuss in Section~\ref{sec:EditingClusterSizes}, after the proof of Proposition~\ref{prop:meetingKrs}. 

We can also consider almost-perfect $H$-tilings, that is, $H$-tilings covering almost all of the vertices of $G$. In the non-partite case the minimum degree condition needed to ensure a tiling covering all but a~linear number of vertices was established by Koml\'os~\cite{K}; this result was later strengthened by Shokoufandeh and Zhao~\cite{SZ} to tilings covering all but a constant number of vertices. Our next theorem provides a~multipartite analogue of the result of Koml\'os.

\begin{theo} \label{theo:almosttiling}
Let $H$ be a graph with $\chi(H) = r \geq 3$. For any $\psi > 0$ there exists $n_0=n_0(\psi,H)$ such that for any $n \geq n_0$, if $G$ is a~balanced $r$-partite graph on $rn$ vertices with $\delta^*(G) \geq \left(1-\frac{1}{\chicr(H)} \right) n$, then $G$ contains an $H$-tiling covering all but at most $\psi n$ vertices of~$G$.
\end{theo}

Bush and Zhao also addressed this problem for the case of $r=2$ and obtained a similar result to Theorem~\ref{BZ:degalpha} -- one without considering the sizes of the connected components -- but their result gives an $H$-tiling covering all but a constant number of vertices.
\begin{theo}[Bush-Zhao~\cite{BZ}] \label{BZ:almosttiling}
For any bipartite graph $H$, there exist constants $n_0=n_0(H)$ and $c=c(H)$ such that whenever $G$ is a~balanced bipartite graph with $n\geq n_0$ vertices in each part,  $\delta(G)\geq (1-1/\chicr(H))n$ suffices to ensure an $H$-tiling of $G$ that covers all but at most $c$ vertices.
\end{theo}

To avoid repetition in proving Theorems~\ref{theo:degalpha} and Theorems~\ref{theo:almosttiling}, we deduce each from the following combined statement.

\begin{theo}\label{theo:combined}
Let $H$ be a graph on $h$ vertices with $\chi(H)=r\geq 3$. For any $\alpha > 0$ there exist $n_0 = n_0(\alpha,H)$ and $C = C(\alpha, H)$ such that the following statements hold for any balanced $r$-partite graph $G$ on $rn$ vertices with $\delta^*(G) \geq \left(1 - \frac{1}{\chicr(H)} + \alpha \right) n$ and $n \geq n_0$.
\begin{enumerate}[(i)]
\item $G$ admits an $H$-tiling covering all but at most $C$ vertices of $G$.
\item If $\gcd(H) = 1$ and $rn$ is divisible by $h$ then $G$ admits a perfect $H$-tiling.
\end{enumerate}
\end{theo}

We prove Theorem~\ref{theo:combined} in Section~\ref{sec:proof} after establishing the necessary preliminaries in Section~\ref{sec:preliminaries}. 
Theorem~\ref{theo:almosttiling} then follows from Theorem~\ref{theo:combined} by a short deduction, which is given in Section~\ref{sec:deduce}, whilst Theorem~\ref{theo:degalpha} is immediate from combining Theorems~~\ref{theo:alonyuster} and \ref{theo:combined} as follows. 

\begin{proof}[Proof of Theorem~\ref{theo:degalpha}]
If $\gcd(H) > 1$ then $\chi^*(H) = \chi(H) = r$, so the existence of such an $n_0$ is given by Theorem~\ref{theo:alonyuster}. On the other hand, if $\gcd(H) = 1$ then $\chi^*(H) = \chicr(H)$, so the existence of such an $n_0$ is given by Theorem~\ref{theo:combined}(ii).
\end{proof}

\subsection{Notation}
We write $x \ll y$ to mean that for any $y > 0$ there exists $x_0 > 0$ such that for any $x$ with $0 < x \leq x_0$ the subsequent statements hold. Similar statements with more variables are defined similarly. We also write $x = y \pm z$ to mean that $y-z \leq x \leq y+z$. We omit floor and ceiling symbols when these do not affect the argument.

\section{Preliminaries} \label{sec:preliminaries}

\subsection{Fractional tilings via linear programming} 
We use the well-known Farkas' Lemma (see \cite[Corollary 7.1d]{Sch}). For this, recall that for a set $Y \subseteq \R^d$ the positive cone $\mathrm{PosCone}(Y)$ of $Y$ is the set of all linear combinations of members of $Y$ with non-negative coefficients.

\begin{theo}[Farkas' Lemma] \label{farkas}
Suppose that $\vb \in \R^d \sm \mathrm{PosCone}(Y)$ for some finite set $Y \subseteq \R^d$.
Then there is some $\xb \in \R^d$ such that $\xb \cdot \yb \leq 0$ for every $\yb \in Y$ and $\xb \cdot \vb > 0$.
\end{theo} 

Let $G$ be a graph on $n$ vertices and let $v_1, \ldots, v_{n}$ be any fixed ordering of its vertices. For a subset of vertices $S$, we denote by $\1_G(S)$ the characteristic vector of $S$, that is, the vector $(x_1,\ldots, x_{n}) \in \mathbb{R}^n$ such that $x_i=1$ for $v_i\in S$ and  $x_i=0$ for $v_i\not\in S$. If $H$ is a subgraph of $G$, we write $\1_G(H)$ instead of~$\1_G(V(H))$. We also write $\1$ for the all-ones vector (the dimension will always be clear from the context).

For a graph $H$, denote by $\K_H(G)$ the set of subgraphs in $G$ isomorphic to $H$.
A \emph{fractional $H$-tiling} in $G$ assigns a weight $w(H') \geq 0$ to each $H'\in\K_H(G)$ such that for any vertex $x \in V(G)$ we have
\begin{align}
   \sum \{w(H') \ |\ H'\in \K_H(G), x\in V(H')\} \leq 1 . \label{eq:tiling}
\end{align}
The fractional $H$-tiling is \emph{perfect} if we have equality in (\ref{eq:tiling}) for every $x \in V(G)$. Equivalently, weights form a fractional $H$-tiling in $G$ if $\sum_{H'\in \K_H(G)} w(H')\,\1_G(H') \leq \1$, where the vector inequality is pointwise, and we have equality if and only if the fractional $H$-tiling is perfect.

Given an integer $r \geq 2$, a \emph{rooted copy of $K_r$} in $G$ is a copy of $K_r$ in which one vertex is designated to be the root. Similarly, given rational numbers $a, b > 0$, an \emph{$(a, b)$-weighted rooted copy of $K_r$} in $G$ is a rooted copy of $K_r$ with the root labelled by $a$ and the remaining vertices by $b$. With every $(a, b)$-weighted rooted copy of $K_r$, $K$, we associate the \emph{weighted characteristic vector} $\1_{a, b, G}(K)$ with $a$ at the coordinate corresponding to the root, $b$ at the other vertices of $K$, and $0$ otherwise. We denote by $\K_{a, b, r}(G)$ the set of all $(a, b)$-weighted rooted copies of $K_r$ in $G$.

This notion extends to the definition of a weighted fractional tiling of $G$:  an \emph{$(a, b)$-weighted fractional $K_r$-tiling} in $G$ consists of a weight $w(K)$ for every $K \in \K_{a, b, r}(G)$ such that $\sum_{K\in\K_{a, b, r}(G)} w(K)\1_{a, b, G}(K) \leq \1$ (where the inequality should be again interpreted pointwise). The tiling is \emph{perfect} if we have equality.

\begin{lem} \label{fractiling}
Let $r\geq 3$ be an integer, let $a$ and $b$ be rational numbers such that $0<a\leq b$ and define $h:= a+(r-1)b$. If $G$ is a balanced $r$-partite graph on $rn$ vertices with $\delta^*(G) \geq (1-b/h)n$ then $G$ admits a perfect $(a, b)$-weighted fractional $K_r$-tiling.
\end{lem}

\proof We will first prove the lemma using the assumption that $bn/h$ is an integer and justify that assumption at the end of the proof. 

Suppose for a contradiction that some graph $G$ as in the statement of the lemma does not admit a~perfect $(a, b)$-weighted fractional $K_r$-tiling, and that $bn/h$ is an integer. This is equivalent to saying that $\1 \notin \mathrm{PosCone}(Y)$ for the set $Y = \{\1_{a, b, G}(K): K \in \K_{a, b, r}(G)\}$. So by Farkas' Lemma, there exists $\xb \in \R^{rn}$ such that 
\begin{align}
   \xb \cdot \1 &> 0 \label{eq:farkasone} 
   \intertext{and}
   \xb \cdot \1_{a, b, G}(K) &\leq 0, \mbox{ for every $K \in \K_{a, b, r}(G)$.} \label{eq:farkasK}
\end{align}
Fix such an $\xb$, and let $v_j^1, \dots, v_j^n$ be the vertices of the $j$th vertex class of $G$,
ordered by decreasing $\xb$-coordinate, that is, so that $\xb \cdot  \1_G(\{v_j^s\}) \geq \xb \cdot  \1_G(\{v_j^t\})$ for any $s \leq t$.

Because $bn/h$ is an integer, each vertex class $V_i$ can be partitioned as follows:
\begin{align}
   V_i^j &:= \{v_i^\ell : (j-1)\, bn/h+1\leq\ell\leq j\, bn/h\}, & \forall j\in [r-1] \label{eq:toprows} \\
   V_i^r &:= \{v_i^\ell : (r-1)\, bn/h+1\leq\ell\leq n\} . & \label{eq:bottomrow}
\end{align}

For any permutation $\pi$ of $[r]$, we can greedily form an $(a, b)$-weighted rooted copy $K_\pi$ of $K_r$ as follows: First, let $u_1$ be the vertex in $V_{\pi(1)}$ for which the $\xb$-coordinate is largest. In our notation, $u_1=v_{\pi(1)}^{1}$. Next, for each $j \in \{2,\ldots,r\}$, let $u_j=v_{\pi(j)}^{t_j}$ be the vertex in $V_{\pi(j)}$ in the common neighborhood of $u_1, \ldots, u_{j-1}$ for which the $\xb$-coordinate is largest. It follows from the minimum degree condition that $t_j \leq (j-1)\, bn/h+1$, so for every $v_{\pi(j)}^\ell \in V_{\pi(j)}^j$ we have $\ell \geq t_j$; in other words every vertex in $V_{\pi(j)}^j$ has $\xb$-coordinate at most that of $u_j$. We assign weight $b$ to each of $u_1,\ldots,u_{r-1}$ and weight $a$ to $u_r$ (so $u_r$ is the root of $K_\pi$).

Since every vertex in $V_{\pi(j)}^j$ has $\xb$-coordinate at most that of $u_j$, we have
\begin{align}
   \xb\cdot\1_G\left(\bigcup_{j=1}^rV_{\pi(j)}^j\right) &\leq \sum_{j=1}^{r-1}\xb\cdot\1_G(\{u_j\}) \frac{bn}{h} \nonumber 
+ \xb\cdot\1_G(\{u_r\})\left(n-(r-1) \frac{bn}{h}\right) \nonumber \\
   &= \sum_{j=1}^{r-1}\xb\cdot\1_G(\{u_j\}) \frac{bn}{h} + \xb\cdot\1_G(\{u_r\}) \frac{an}{h} \nonumber \\
   & = \xb\cdot\left(\frac{n}{h} \1_{a,b,G}(K_\pi)\right) . \label{eq:xdomination}
\end{align}
This gives the following contradiction
\begin{align*}
   0 < (r-1)!\xb \cdot \1 = \sum_\pi \xb\cdot\1_G\left(\bigcup_{i=1}^r V_{\pi(j)}^j\right)\leq \frac{n}{h}\sum_\pi \xb\cdot\1_{a,b,G}(K_\pi)\leq 0,
\end{align*}
where each sum is taken over all permutations $\pi$ of $[r]$. The equality in this calculation  is due to the fact that the sets $V_i^j$ partition $V(G)$, and for any $i, j \in [r]$ there are precisely $(r-1)!$ permutations of $[r]$ with $\pi(j) = i$. The first inequality follows from (\ref{eq:farkasone}), the second inequality follows from (\ref{eq:xdomination}), and the final inequality follows from (\ref{eq:farkasK}).

In order to complete the proof, we justify the assumption that $bn/h$ is an integer. To see this, fix an integer $m$ such that $bnm/h$ is an integer, and let $G'$ be the \emph{$m$-fold blow-up of $G$}, in which each vertex $v \in V(G)$ is replaced by $m$ copies of $v$ in $G'$, and each edge $uv \in E(G)$ is replaced by $m^2$ edges between the copies of $u$ and $v$ in $G'$. Also set $n' := nm$. Then $G'$ is a balanced $r$-partite graph on $rn'$ vertices with $\delta^*(G') = m\delta^*(G) \geq (1-b/h)n'$, and $bn'/h = bnm/h$ is an integer. 

Given that the lemma holds in this case, $G'$ admits a perfect $(a, b)$-weighted fractional $K_r$-tiling. This naturally yields a perfect $(a, b)$-weighted fractional $K_r$-tiling in $G$ by taking the weight of each rooted copy of $K_r$ in $G$ to be the average of the weights of the $m^r$ corresponding rooted copies of $K_r$ in $G'$.
\endproof

\begin{rem} \label{rem:rationalweights}
A perfect $(a, b)$-weighted fractional $K_r$-tiling as guaranteed by \linebreak Lemma~\ref{fractiling} is the solution to a linear programming instance in which all coefficients are rational. Such an instance must have a rational solution, so we may assume that all weights in a perfect $(a, b)$-weighted fractional $K_r$-tiling given by Lemma~\ref{fractiling} are rational.
\end{rem}

\subsection{Editing cluster sizes} 
\label{sec:EditingClusterSizes}
Proposition~\ref{defU} below shows that we may `combine' copies of $H$ to form a complete $r$-partite graph $\U(H)$ whose vertex classes are all equal except for one class which has one extra vertex and one class which has one fewer vertex. This will allow us, in the proof of Theorem~\ref{theo:combined}, to delete copies of $\U(H)$ and thus modify the sizes of clusters of $H$ modulo $rh$.

\begin{prop} \label{defU}
Let $H$ be a graph on $h$ vertices with $\chi(H) = r \geq 3$ and $\gcd(H) = 1$. Then there exists an integer $s = s(H)$ for which the complete $r$-partite graph $\U(H)$ with one vertex class of size $srh+1$, one vertex class of size $srh-1$ and $r-2$ vertex classes of size $srh$ admits a perfect $H$-tiling.
\end{prop}

The proof of Proposition~\ref{defU} is straightforward and essentially identical to that of Proposition~3.6 from~\cite{M} (which gave an analogous statement for $r$-partite $r$-uniform hypergraphs $H$), so we omit it. To apply Proposition~\ref{defU} we make use of the following elementary proposition, which we will apply in the `reduced graph', and then apply Proposition~\ref{defU} within the graph induced by the clusters corresponding to $K$ and $K'$.

\begin{prop} \label{prop:meetingKrs}
Fix $r \geq 3$, and let $G$ be a balanced $r$-partite graph on $rn$ vertices with $\delta^*(G) > (1-\frac{1}{r-1})n$. Then for any vertices $u, v \in V(G)$ there are copies $K$ and $K'$ of $K_r$ in $G$ such that $u \in K$, $v \in K'$, and such that $K$ and $K'$ have at least one vertex in common.
\end{prop}

\begin{proof}
Let $V_1, \dots, V_r$ be the vertex classes of $G$, and assume without loss of generality that $u, v \notin V_r$. Since $r \geq 3$ we have $\delta^*(G) > n/2$, so we may fix a common neighbour $w$ of $u$ and $v$ in $V_r$. It then suffices to extend $\{u, w\}$ and $\{v, w\}$ to copies of $K_r$ in $G$, and we may do this greedily. Indeed, any set $S$ of $j$ vertices of $G$ has at least $n-j(n-\delta^*(G)) > n - \frac{jn}{r-1}$ common neighbours in each vertex class not intersected by $S$, and in forming a copy of $K_r$ we choose each vertex to be a common neighbour of at most $r-1$ previously-chosen vertices, so there is always a common neighbour available.
\end{proof}

Observe that the statement of Proposition~\ref{prop:meetingKrs} does not hold for $r=2$, as then $G$ need not be connected. This is the fundamental reason for the different behaviour of Theorem~\ref{theo:degalpha} compared to Theorem~\ref{BZ:degalpha} (in which the greatest common divisor of the sizes of connected components plays a role).

\subsection{Completing the tiling} 
At the end of the proof of Theorem~\ref{theo:combined}, all the remaining vertices of our graph $G$ lie in vertex-disjoint $r$-partite subgraphs $G''$ of $G$ whose vertex classes are pairwise super-regular with positive density. We then complete the $H$-tiling of $G$ by finding a~perfect $H$-tiling of each $G''$. For this it would be natural to arrange that the sizes of the $r$ vertex classes of $G''$ are in the ratio $b:b:\dots:b:a$, so that the proportion of vertices of $G''$ in the smallest vertex class is the same as the proportion of vertices of $H$ in the smallest vertex class.
But it is to our advantage to ensure that the smallest class of $G''$ actually has a slightly larger proportion of the vertices. Indeed, Proposition~\ref{completetiling} below guarantees that such a distribution of sizes (together with certain divisibility assumptions) ensures a perfect $H$-tiling in the complete $r$-partite graph $G'$ with the same vertex class sizes. This is enough for the Blow-up Lemma to ensure that $G''$ also admits a perfect $H$-tiling.

\begin{prop}[\cite{M}, Corollary~6.13] \label{completetiling}
Let $H$ be a graph on $h$ vertices with $\chi(H) = r \geq 3$ and $\sigma(H) < \frac{1}{r}$. Then for any $\alpha > 0$ there exist $\beta = \beta(\alpha,H) > 0$ and $n_0 = n_0(\alpha,H)$ such that the following statement holds.
 
Let $G'$ be a complete $r$-partite graph on $n \geq n_0$ vertices with vertex classes $V_1, \dots, V_r$, where $|V_1| \leq |V_2|, \cdots, |V_r|$. Suppose also that
\begin{enumerate}[(1)]
\item $\sigma(G') \geq \sigma(H) + \alpha$,
\item $\big||V_i| - |V_j|\big| \leq \beta n$ for any $2 \leq i, j \leq r$, and
\item $rh\cdot\gcd(H)$ divides $|V_j|$ for each $j \in [r]$.
\end{enumerate}
Then $G'$ admits a perfect $H$-tiling.
\end{prop}
Proposition~\ref{completetiling} is also taken from~\cite{M}, where it was stated for $r$-partite $r$-uniform hypergraphs $H$; here it is easy to see that the $r$-partite graph form is identical, since the problem of tiling a complete $r$-partite $r$-uniform hypergraph $G'$ with copies of a smaller $r$-partite $r$-uniform hypergraph $H$ is identical to the problem of tiling a complete $r$-partite graph $G'$ with copies of a smaller $r$-partite graph $H'$.

\subsection{Tidying up atypical vertices} 
In the proof of Theorem~\ref{theo:combined} we will encounter `bad' vertices in $G$ which have atypical neighbourhoods. At an early stage in the proof we will greedily remove each such vertex $v$ from $G$ by deleting a copy of $H$ in $G$ which contains $v$. The following proposition shows that the degree condition of Theorem~\ref{theo:combined} is (more than) strong enough to ensure that this is possible.

\begin{prop} \label{deletevertex}
Let $H$ be a graph on $h$ vertices with $\chi(H) = r \geq 3$ and $\gcd(H) = 1$. For any $\alpha > 0$, there exists $n_0=n_0(\alpha,H)$ such that the following statement holds.
 
Let $G$ be a balanced $r$-partite graph on $rn$ vertices such that $n \geq n_0$ and $\delta^*(G) \geq \frac{r-2}{r-1}n + \alpha n$. Then, for any vertex $v \in V(G)$, there is a copy of $H$ in $G$ which contains $v$.
\end{prop}

We omit any proof of Proposition~\ref{deletevertex} in that it is a straightforward application of Szemer\'edi's Regularity Lemma and is implicit in many papers, including~\cite{BZ,KO2,MS}.

\subsection{The regularity method} 
We use a variant of Szemer\'edi's Regularity Lemma.  Before we can state it, we need a few basic definitions. For disjoint vertex sets $A$ and $B$ in some graph, let $e(A,B)$ denote the number of edges with one endpoint in $A$ and the other in $B$.  Further, let the \textit{density} of the pair $(A,B)$ be $d(A,B)=e(A,B)/|A||B|$. We say that the pair $(A,B)$ is \textit{$\eps$-regular} if $X\subseteq A$, $Y\subseteq B$, $|X|\geq\eps |A|$, and $|Y|\geq\eps|B|$ imply $|d(X,Y)-d(A,B)|\leq\eps$, and likewise that a pair $(A,B)$ is \textit{$(\eps, \delta)$-super-regular} if $(A, B)$ is $\eps$-regular and also $\deg_B(a)\geq\delta |B|$ for all $a\in A$ and $\deg_A(b)\geq\delta |A|$ for all $b\in B$.

The degree form of Szemer\'edi's Regularity Lemma (see, for instance, \cite[Theorem 1.10]{KS}) is sufficient here, modified for the multipartite setting.

\begin{theo}\label{thm:SzemRegLem}
For every integer $r \geq 2$ and every $\eps>0$, there is an $M=M(r,\eps)$ such that if $G=(V_1,\ldots,V_r;E)$ is a balanced $r$-partite graph on $rn$ vertices and $d\in[0,1]$ is any real number, then there exist integers $\ell$ and $L$, a spanning subgraph $G'=(V_1,\ldots,V_r;E')$ and, for each  $i=1,\ldots,r$, a partition of $V_i$ into clusters $V_i^{0},V_i^{1},\ldots,V_i^{\ell}$ with the following properties.
\begin{enumerate}[(P1)]
   \item $\ell\leq M$,
   \item $|V_i^{0}|\leq \eps n$ for $i\in [r]$,
   \item $|V_i^{j}|=L\leq\eps n$ for $i\in [r]$ and $j\in [\ell]$,
   \item $\deg_{G'}(v,V_{i'})>\deg_G(v,V_{i'})-(d+\eps)n$ for all $v\in V_i$, $i\neq i'$, and
   \item all pairs $(V_i^{j},V_{i'}^{j'})$, $i,i'\in [r]$, $i\neq i'$, $j,j'\in [\ell]$, are $\eps$-regular in $G'$, each with density either $0$ or exceeding $d$. \label{it:P5}
\end{enumerate}
\end{theo}

The final step in the proof of Theorem~\ref{theo:combined} is to apply the Blow-up Lemma of  Koml\'os, S\'ark\"ozy, and Szemer\'edi~\cite{KSSz97} in the following form. 

\begin{theo}[Blow-up Lemma]\label{theo:blow-up}
For any integers $r$ and $\Delta$ and any $\delta > 0$ there exist $\eps = \eps(r, \Delta, \delta) > 0$ and $N_0 = N_0(r, \Delta, \delta)$ such that the following holds for any integer $N \geq N_0$ and any graph $R$ on vertex set~$[r]$. 

Let $V_1, \dots, V_r$ be pairwise-disjoint sets each of size $N$, and set $V = \bigcup_{i \in [r]} V_i$. Let $K$ be the graph on vertex set $V$ in which $(V_i, V_j)$ is a~complete bipartite graph for $ij \in E(R)$ (and which has no other edges than these). Also let $G$ be any graph on $V$ in which $(V_i, V_j)$ is $(\eps, \delta)$-super-regular for any $ij \in E(R)$. Then for any graph $H$ with maximum degree $\Delta(H) \leq \Delta$, if $H$ can be embedded in $K$, then $H$ can also be embedded in $G$.
\end{theo}
This essentially states that we may treat super-regular pairs as being complete for the sake of embedding bounded degree spanning subgraphs (such as a perfect $H$-tiling).

\section{Proof of Theorem~\ref{theo:combined}} \label{sec:proof}

We now give the full proof of Theorem~\ref{theo:combined}. Recall that $r=\chi(H) \geq 3$ and $h = |V(H)|$, and set $\sigma := \sigma(H)$, $a := \sigma h$ and $b := (1-\sigma)h/(r-1)$. 
Then $a, b$ and $\sigma$ are positive rational numbers with $h = a + (r-1)b$ and $\sigma \leq \tfrac{1}{r}$ (see~\eqref{eq:defsigma}). If $\sigma = \frac{1}{r}$ then $\chicr(H) = r$ by definition of $\chicr$ (see~\eqref{eq:defchicr}), so Theorem~\ref{theo:alonyuster} gives the theorem in this case. We may therefore assume that $\sigma < \tfrac{1}{r}$. Since both $\sigma$ and $\tfrac{1}{r}$ can be written as rationals with denominator $rh$ it follows that $\sigma \leq \tfrac{1}{r}-\tfrac{1}{rh}$, so $b-a \geq \tfrac{1}{r-1}$. Without loss of generality we assume that $\alpha$ is rational and that $\alpha \leq \tfrac{1}{rh}$. Introduce new constants $n_0, C, D, M, \eps, \eps', \beta, d$ with
$$ \tfrac{1}{n_0} \ll \tfrac{1}{C} \ll \tfrac{1}{D} \ll \tfrac{1}{M} \ll \eps \ll \eps' \ll \beta \ll d \ll \alpha, \tfrac{1}{r}, \tfrac{1}{h}.$$
Let $G$ be an $r$-partite graph whose vertex classes $V_1, \dots, V_r$ each have size $n \geq n_0$ and which satisfies 
$$\delta^*(G) \geq \left(1-\frac{1}{\chicr(H)}  + \alpha \right) n \stackrel{\eqref{eq:defchicr}}{=} \left(1 - \frac{1-\sigma}{r-1} + \alpha \right) n= \left(1-\frac{b}{h}\right) n + \alpha n.$$
We shall construct an $H$-tiling in $G$ covering all but at most $C$ vertices of $G$, or, if $\gcd(H) = 1$ and $h$ divides $rn$, a perfect $H$-tiling in $G$.

Define $a' := a + \frac{\alpha h}{2}$ and $b' := b - \frac{\alpha h}{2(r-1)}$, so $a'$ and $b'$ are rational numbers with $0 < a' \leq b'$ (the latter inequality follows from our assumption that $\alpha \leq \frac{1}{rh}$) and $ a' + (r-1)b' = a + (r-1)b = h$. Note also that 
$1 -\frac{b'}{h} \leq 1 - \frac{b}{h} + \frac{\alpha}{2(r-1)}  \leq 1 - \frac{b}{h} + \frac{\alpha}{2}$,
so 
\begin{equation} \label{eq:mindegG}
\delta^*(G) \geq \left(1 - \frac{b'}{h}\right) n + \frac{\alpha n}{2}.
\end{equation}

\medskip \noindent \emph{Step 1: Apply the Regularity Lemma and define the reduced graph $R$.}
We apply the Regularity Lemma (Theorem~\ref{thm:SzemRegLem}) to $G$, with $r$, $\eps$, $d$ and $M$ playing the same role there as here, to obtain integers $\ell$ and $L$, a spanning subgraph $G'$ of $G$ and a partition of each $V_i$ into clusters $V_i^0, V_i^1, \dots, V_i^\ell$ which satisfy properties (P1)-(P5). In particular, (P3) tells us that for any $i \in [r]$ and $j \in [\ell]$ the cluster $V_i^j$ has size $L$, so $(1-\eps)n/\ell \leq L \leq n/\ell$. We define the \emph{reduced graph} $R$ of $G'$ in a standard way: the vertices of $R$ are the clusters $V_i^j$ for $i \in [r]$ and $j \in [\ell]$, and the edges of $R$ are those $V_{i}^{j}V_{i'}^{j'}$ for which there is at least one edge of $G'$ between $V_i^j$ and $V_{i'}^{j'}$ (note that (P5) then implies that the pair $(V_i^j, V_{i'}^{j'})$ is $\eps$-regular with density at least $d$). So $R$ is $r$-partite with vertex classes of size $\ell$. Moreover, for any $i \neq j$, any vertex $v \in V_i$ has at least $\delta^*(G) - (d+\eps)n$ neighbours in $V_j$ by (P4). By (P2) at most $\eps n$ of these neighbours are in $V_j^0$, so $v$ has neighbours in at least $\tfrac{1}{L} \cdot (\delta^*(G) - (d+2\eps)n)$ of the clusters $V_j^1, \dots, V_j^\ell$. Since $L \leq n/\ell$, it follows from (\ref{eq:mindegG}) that 

\begin{align} \label{eq:mindegR}
\delta^*(R) & \geq \frac{\ell}{n}\left(\left(1 - \frac{b'}{h} \right)n + \frac{\alpha n}{2} - (d+2\eps)n \right) 
\geq \left(1 - \frac{b'}{h}\right)\ell. 
\end{align} 

\medskip \noindent\emph{Step 2: Obtain a perfect fractional $(a', b')$-weighted $K_r$-tiling $\T$ in~$R$.} 
This can be done immediately by applying Lemma~\ref{fractiling} to~$R$ (inequality (\ref{eq:mindegR}) tells us that the minimum degree condition is satisfied). Let $\K^+$ be the set of $(a', b')$-weighted rooted copies of $K_r$ of non-zero weight in $\T$, that is, $\K^+ = \{K \in \K_{a', b', r}(R) : w(K) > 0\}$. Also observe that $R$ has $r\ell \leq rM$ vertices, so the number of possibilities for the reduced graph $R$ is bounded by a function of $M$. For each possible $R$, Lemma~\ref{fractiling} would give us a perfect fractional $(a', b')$-weighted $K_r$-tiling of $R$ in which all weights are rational (see Remark~\ref{rem:rationalweights}). So, as observed in Section 3.2 from~\cite{MS}, there is a common denominator, bounded by a function of $M$, of all weights used in our perfect fractional $(a', b')$-weighted $K_r$-tilings for each possible reduced graph $R$. Since $1/D \ll 1/M$, we may assume that $D!$ is a multiple of this common denominator, and therefore that $w(K)D!$ is an integer for any $K \in \K^+$. In particular, $w(K) \geq 1/D!$ for every $K \in \K^+$.

\medskip \noindent\emph{Step 3: Partition the clusters $U_i$ into subclusters according to the fractional tiling $\T$.} For each $i \in [r]$ and $j \in [\ell]$ let $\K^+_{i, j}$ consist of all members of $\K^+$  which contain $V^j_i$. So each member of $\K^+$ appears in precisely $r$ of the sets $\K^+_{i, j}$. Also, since $\T$ is perfect, for any cluster $V^j_i$ we have
$$\sum_{K \in \K^+} w(K) \1_R(\{V^j_i\}) \cdot \1_{a', b', R}(K) = \sum_{K \in \K^+_{i,j}} w(K) \1_R(\{V^j_i\}) \cdot \1_{a', b', R}(K) = 1 . $$
Recall that $\1_{a', b', R}(K)$ is the vector where the entries are $a$ at the coordinate corresponding to the root, $b$ at the other vertices of $K$, and $0$ otherwise. So we may partition the cluster $V^j_i$ into parts ${V^j_i}(K)$ for $K \in \K_{i,j}^+$ such that $|{V^j_i}(K)| = w(K)L \1_R(\{V^j_i\}) \cdot \1_{a', b', R}(K)$; we refer to these parts as \emph{subclusters}. Having partitioned each cluster in this manner, for each $K \in \K^+$ we collect together the corresponding $r$ parts $V^{i}_{j}(K)$. One of these parts (taken from the root of $K$) has size $a'w(K) L$, and we relabel this subcluster as $U_1^K$; the remaining $r-1$ parts have size $b' w(K) L$, and we relabel these subclusters as $U_2^K, \dots, U_r^K$. For each $K \in \K^+$ define $m_1^K := a'w(K)L$ and $m_i^K :=b'w(K)L$ for $2 \leq i \leq r$, so that each subcluster $U_i^K$ has size $m_i^K$.

We refer to the cluster from which a subcluster is taken as the \emph{parent cluster} of that subcluster. Moreover, we choose the partition into subclusters in such a way that whenever $U_i^K$ and $U_{j}^{K'}$ are subclusters whose parent clusters form an edge of $R$, the pair $(U_i^K, U_{j}^{K'})$ is $\eps'$-regular in $G'$ with density $d(U_i^K, U_{j}^{K'}) \geq d/2$. This is possible since each subcluster has size at least $a' w(K) L \geq a'L/D!$. Indeed, the Random Slicing Lemma (see e.g.~\cite[Lemma 10]{MS}) states that the described event holds with high probability if we choose the partition of each cluster uniformly at random.

For each $K \in \K^+$ let $G^K$ denote the subgraph of $G'$ induced by $U^K := \bigcup_{i \in [r]} U_i^K$. So $G^K$ is naturally $r$-partite with vertex classes $U_i^K$ for $i \in [r]$. Furthermore, the graphs $G^K$ for $K \in \K^+$ are vertex-disjoint and collectively cover all vertices of $G$ other than those in the sets $V_i^0$ for $i \in [r]$. Over the next three steps of the proof we will remove or delete some vertices from each subcluster $U_i^K$; whenever we do so we continue to write $U_i^K$, $U^K$ and $G^K$ for the restriction of these sets/graphs to the vertices which were not removed or deleted. Note, however, that we do not edit the quantities $m_i^K$, $L$ and $n$ as vertices are removed or deleted.

\medskip \noindent\emph{Step 4: Remove some vertices to make each $G^K$ super-regular.} For each $K \in \K^+$ and $i \in [r]$ we say that a vertex $v \in U_i^K$ is \emph{bad} if $|N_{G'}(v) \cap U_j^K| < (d/2-\eps')m_j^K$ for some $j \neq i$. By our choice of partition of clusters into subclusters, $(U_i^K,U_j^K)$ is an $\eps'$-regular pair in $G'$ with $d(U_i^K,U_j^K) \geq d/2$ for each $j \neq i$, so there are at most $(r-1)\eps'm_i^K$ bad vertices in $U_i^K$. We now remove all bad vertices from $U_i^K$ for each $K \in \K^+$ and $i \in [r]$. 

Let the set $X$ consist of all removed vertices and also the vertices of $V_i^0$ for each $i \in [r]$, so $|X| \leq (r-1) \eps' n + r \eps n \leq r\eps' n$, and the set $X$ and subclusters $U_i^K$ partition $V(G)$. Moreover, since all bad vertices were removed, for each $K \in \K^+$ and each $i \neq j$ the pair $(U_i^K, U_j^K)$ is now $(2\eps', d/3)$-super-regular.

At this point we note that over the next two steps of the proof at most $2\beta m_i^K + C$ vertices will be deleted from each subcluster $U_i^K$, in addition to the at most $(r-1)\eps' m_i^K$ vertices removed during the current step. Since $C \leq \frac{1}{h} \cdot \frac{(1-\eps)n}{\ell} \cdot \frac{1}{D!} \leq a'Lw(K) \leq \eps m_i^K$, this means that in total at most $3\beta m_i^K \leq \frac{d}{12} m_i^K$ vertices are removed or deleted from $U_i^K$, and so even after some or all of these deletions it will remain the case that
\begin{enumerate}
\item[(S1)] If $W_1$ and $W_2$ are subclusters whose parent clusters form an edge of $R$, then $(W_1, W_2)$ is a $2\eps'$-regular pair in $G'$ with density at least $d/3$.
\item[(S2)] For any $K \in \K^+$ and $i \neq j$ the pair $(U_i^K,U_j^K)$ is $(3\eps', d/4)$-super-regular in $G'$.
\end{enumerate}

\medskip \noindent\emph{Step 5: Delete copies of $H$ which cover all vertices of $X$.}
 We now delete at most $|X|+r$ vertex-disjoint copies of $H$ from $G$ so that every vertex of $X$ is deleted, at most $2 \beta m_i^K$ vertices are deleted from any subcluster $U_i^K$, and also, if $h$ divides $rn$, so that the total number of undeleted vertices is divisible by $rh$. This can be done greedily. Indeed, since in total we choose at most $|X| + r \leq 2r \eps' n$ copies of $H$, at most $2r \eps' n h$ vertices are deleted in total.  
 
Prior to any deletion, we `mask' any vertices in any subcluster $U_i^K$ from which at least $\beta m_i^K$ vertices (i.e. at least a $\beta$-proportion of the vertices) have previously been deleted; there are then at most $2r \eps' nh/\beta \leq \beta n$ vertices which lie in masked subclusters. Together with the at most $2r \eps' n h \leq \beta n$ vertices in copies of $H$ already deleted in this step, this means we must choose the next copy of $H$ so as to avoid at most $2\beta n$ vertices of $G$. So the restriction of $G$ to the as-yet-undeleted vertices of $G$ has minimum multipartite degree at least $\frac{r-2}{r-1}n + \alpha n$ (recall from \eqref{eq:defchicr} that $\chicr(H) > \chi(H)-1 = r-1$). We may therefore select any as-yet-undeleted vertex $v$ and apply Proposition~\ref{deletevertex} to obtain a copy of $H$ within this restriction which contains $v$, which we then delete. Whilst $X$ remains non-empty we always choose $v \in X$, which ensures that after at most $|X|$ deletions every vertex of $X$ will have been deleted. 
 
If $h$ does not divide $rn$ we are then done, so suppose now that $h$ divides $rn$. We continue as before, now choosing $v$ at each step to be an arbitrary unmasked vertex. Since each time we delete a copy of $H$ we delete $h$ vertices from $G$, the number of undeleted vertices of $G$ is always divisible by $h$, and so we can ensure that the number of undeleted vertices of $G$ is divisible by $rh$ by deleting at most a further $r-1$ copies of $H$, as claimed. Finally, the fact that masked vertices cannot be deleted ensures that at most $\beta m_i^K + h \leq 2 \beta m_i^K$ vertices are deleted from any subcluster $U_i^K$, as required.

\medskip \noindent \emph{Step 6: Delete vertices or copies of $H$ from $G$ to ensure divisibility of subcluster sizes.}
For (i) of Theorem~\ref{theo:combined}, in which we only wish to find an $H$-tiling covering all but at most $C$ vertices of $G$, we now simply delete vertices of $G$ individually so that, following these deletions, the size of each subcluster is divisible by $rh\cdot\gcd(H)$ (the deleted vertices will not be covered by the $H$-tiling we construct). Since we have $r\ell \leq rM$ clusters, each of which was partitioned into at most $D!$ subclusters, we can achieve this by deleting at most $rMD! \cdot rh\cdot\gcd(H) \leq C$ vertices. These are the only vertices of $G$ which will not be covered by the $H$-tiling we are constructing. 

Now consider (ii), in which we assume that $\gcd(H) = 1$ and that $h$ divides $rn$. By Proposition~\ref{defU}, we may choose an integer $s$ for which the complete $r$-partite graph $\U(H)$ with vertex classes $Y_1, Y_2, \dots, Y_r$ of sizes $|Y_1| = srh+1$, $|Y_2| = \dots = |Y_{r-1}| = srh$ and $|Y_r| = srh-1$ admits a perfect $H$-tiling. Moreover, since $s$ depends only on $H$, and $1/M \ll 1/h$, we may assume that $s \leq M$. We now delete vertex-disjoint copies of $\U(H)$ from $G$ so that, following these deletions, the size of each subcluster is divisible by $rh$ (since $\U(H)$ admits a perfect $H$-tiling, deleting a copy of $\U(H)$ from $G$ is equivalent to deleting $sr^2$ vertex-disjoint copies of $H$ from $G$). We do this by iterating the following steps. 

If every subcluster has size divisible by $rh$, then we are done. Otherwise, since the total number of undeleted vertices is divisible by $rh$, there must be two subclusters $W_1$ and $W_1'$ whose size is not divisible by $rh$. Let $X_1$ and $X_1'$ be the parent clusters of $W_1$ and $W_1'$ respectively. Then by (\ref{eq:mindegR}) and Proposition~\ref{prop:meetingKrs} we may choose clusters $X_2, \dots, X_r$ and $X'_2, \dots, X'_{r-1}$ such that $\{X_1, X_2, \dots, X_r\}$ and $\{X'_1, X'_2 \dots, X'_{r-1}, X_r\}$ each induce copies of $K_r$ in $R$. Arbitrarily choose subclusters $W_2, \dots, W_r$ and $W'_2, \dots, W'_{r-1}$ such that $X_i$ and $X_i'$ are the parent clusters of $W_i$ and $W_i'$ respectively. Now let $z \in [rh-1]$ be such that $|W_1| \equiv z$ modulo $rh$. Greedily choose and delete $z$ vertex-disjoint copies of $\U(H)$ in $G$ in which $Y_i$ is embedded to $W_i$ for each $i \in [r]$. Having done so, greedily choose and delete a further $z$ vertex-disjoint copies of $\U(H)$ in $G$ in which $Y_1$ is embedded to $W_r$, $Y_r$ is embedded to $W'_1$, and $Y_i$ is embedded to $W'_i$ for each $2 \leq i \leq r-1$ (we shall explain shortly why it is possible to choose copies of $\U(H)$ in this way). Then, modulo $rh$, the effect of these deletions is to reduce $|W_1|$ by $z$, to increase $|W'_1|$ by $z$, and to leave the size of all other subclusters unchanged. So $W_1$ now has size divisible by $rh$, and so the number of subclusters whose size is not divisible by $rh$ has been reduced by at least $1$. At this point we proceed to the next round of the iteration.

Since there are at most $r M D!$ subclusters, this process terminates after at most $r M D!$ iterations, at which point each subcluster has size divisible by $rh$. In each iteration we delete fewer than $2rh$ copies of $|\U(H)|$, each of which has $sr^2h \leq Mr^2h$ vertices, so in total at most $r M D! \cdot 2rh \cdot Mr^2h \leq C$ vertices are deleted in this step. 

It remains only to explain why it is always possible to choose copies of $\U(H)$ as desired. To see this, suppose that we have already deleted copies of $\U(H)$ covering up to $C$ vertices of $G$, and that we next wish to choose and delete a copy of $\U(H)$ within subclusters $W_1, \dots, W_r$ whose parent clusters $X_1, \dots, X_r$ form a copy of $K_r$ in $R$. It follows from (S1) that at this point $(W_i, W_j)$ is a $2\eps'$-regular pair in $G'$ of density at least $d/3$ for each $i \neq j$. The fact that $n \geq n_0$ is sufficiently large implies that each subcluster $W_i$ is large enough to apply the Counting Lemma (see, e.g.,~\cite{RS}), which guarantees that a copy of $\U(H)$ can be found in $G'[\bigcup_{i \in [r]} W_i]$, with vertex classes embedded in the desired manner.

Observe that since at most $2 \beta m_i^K$ vertices were deleted from any subcluster $U_i^K$ in Step 5, and at most $C$ vertices were deleted in total in this step, the total number of vertices deleted from any subcluster is at most $2\beta m_i^K + C \leq 3\beta m_i^K$, justifying our assertion at the end of Step 4.

\medskip \noindent \emph{Step 7: Blow-up a perfect $H$-tiling in each $G^K$.} Consider any $K \in \K^+$. Recall that prior to any removals or deletions each subcluster $U_i^K$ had size $m_i^K$, where $m_1^K = a'w(K)L$ and $m_i^K = b'w(K)L$ for $2 \leq i \leq r$.
Since then we have removed or deleted at most $3\beta m_i^K$ vertices (i.e. at most a $3\beta$-proportion) from each $U_i^K$, so in particular (since $b' \geq a'$) each subcluster $U_i^K$ now has size at least $m_1^K - 3 \beta m_1^K$. So if we let $\hat{G}^K$ denote the complete $r$-partite graph whose vertex classes are the subclusters $U_1^K, \dots, U_r^K$, then we now have 
\begin{align*}
\sigma(\hat{G}^K) & = \frac{\min_{i \in [r]} |U_i^K|}{|U^K|} \geq \frac{m_1^K - 3\beta m_1^K}{\sum_{i=1}^{r} m_i^K} \geq  \frac{(1-3\beta) a' w(K)L}{(a' + (r-1)b') w(K)L}  \\ 
&= (1-3\beta) \frac{a'}{h} \geq \frac{a}{h} + \frac{\alpha}{2} - 3\beta \geq \sigma + \frac{\alpha}{3}.
\end{align*}
Also, we now have $|U^K| \geq (1-3\beta) \sum_{i=1}^r m_i^K \geq m_2^K$, so for any $2 \leq i, j \leq k$ we have $||U_i^K| - |U_j^K|| \leq 3 \beta m_2^K \leq 3 \beta |U^K|.$
Since our deletions in Step~6 ensured that $rh\cdot\gcd(H)$ now divides $|U_i^K|$ for each $i \in [r]$, the graph $\hat{G}^K$ satisfies the conditions of Proposition~\ref{completetiling} (with $\alpha/3$, $3 \beta$ and $|U^K|$ in place of $\alpha$, $\beta$ and $n$ respectively, with the smallest subcluster $U_i^K$ in place of $V_1$, and the remaining subclusters in place of $V_2, \dots, V_r$). By this proposition $\hat{G}^K$ contains a perfect $H$-tiling. Since by (S2) each pair $(U_i^K, U_j^K)$ is $(3\eps', d/4)$-super-regular in $G'$, the Blow-up Lemma (Theorem~\ref{theo:blow-up}) implies that there is also a perfect $H$-tiling $M^K$ in $G^K$. Let $M^*$ be the $H$-tiling in $G$ consisting of all the copies of $H$ which were deleted in Steps~5 and~6. Then $M := M^* \cup \bigcup_{K \in \K^+} M^K$ is an $H$-tiling in $G$ which covers all vertices of~$G$ except the at most $C$ vertices deleted individually in Step~6, proving (i). For (ii) recall that in this case no vertices were deleted individually in Step~6, so $M$ is a perfect $H$-tiling in~$G$. \qed

\section{Proof of Theorem~\ref{theo:almosttiling}} \label{sec:deduce}

The proof of Theorem~\ref{theo:almosttiling} is an immediate corollary of Theorem~\ref{theo:combined}. Indeed, fix $0 < \psi \leq 1$, and let $H$ be a graph on $h$ vertices with $\chi(H) = r \geq 3$. Set $k: =\chicr(H)$ and $\alpha : = \tfrac{\psi}{2rkh}$, and take $C$ and $n_0$ large enough to apply Theorem~\ref{theo:combined} and such that $C \leq \alpha n_0$. Consider a balanced $r$-partite graph $G$ on $rn$ vertices with $\delta^*(G)\geq\frac{k-1}{k}\, n$ and $n \geq n_0$.

We construct an auxiliary graph $G'$ from $G$ by adding  the same number $m$ of dummy vertices to each vertex class, where $m := 2k\alpha n \leq n$. We make these dummy vertices adjacent to every other vertex, except vertices in their own vertex class. As a result, $G'$ is a balanced $r$-partite graph on $rn'$ vertices with $n'=n+m$ and 
\begin{align*}
 \delta^*(G')=\delta^*(G)+m & \ge \frac{k-1}{k}n+ m = \frac{k-1}{k}(n+m) + \frac{m}{k} \\ &= \frac{k-1}{k}n' + 2 \alpha n \ge \left(\frac{k-1}{k}+\alpha\right) n'. 
\end{align*}
So we may apply Theorem~\ref{theo:combined}(i) to $G'$ to obtain an $H$-tiling of $G'$ which covers all but at most $C$ vertices of $G'$. There are at most $rm$ copies of $H$ in this tiling that contain a dummy vertex. We remove these copies of $H$ to obtain an $H$-tiling of $G$ that covers all but at most $rm(h-1) + C \leq 2rk\alpha (h-1) n + \alpha n \leq \psi n$ vertices of~$G$.
\qed

\section{Lower bound constructions}
\label{sec:construct}

In this section we present simple constructions which show that the minimum degree condition of Theorem~\ref{theo:degalpha} is best-possible up to the error term. These are all variations of the following construction.

\begin{construct}\label{construct:general}
Let $r$, $n$ and $n_{ij}$ for $i, j \in [r]$ be positive integers with \linebreak $\sum_{j \in [r]} n_{ij} = n$ for each $i \in [r]$. Choose pairwise-disjoint sets $V_i^j$ with $|V_i^j| = n_{ij}$ for each $i,j \in [r]$. Let $G = G((n_{ij}), r)$ be the graph with vertex set $\bigcup_{i, j \in [r]} V_i^j$ and in which the pairs $(V_i^j,V_{i'}^{j'})$ induce complete bipartite graphs whenever both $i\neq i'$ and $j\neq j'$ (and no other edges exist). We refer to the sets $V_i^j$ as \emph{blocks}, to the sets $V_i := \bigcup_{j \in [r]} V_i^j$ as \emph{columns} and to the sets $V^j := \bigcup_{i \in [r]} V_i^j$ as \emph{rows}. So each vertex is adjacent to every other vertex which is not in the same row or column. Moreover we view $G$ as a balanced $r$-partite graph whose vertex classes are the columns $V_i$ for $i \in [r]$, so each vertex class $V_i$ has size $|V_i| = \sum_{j \in [r]} n_{ij} = n$. Observe that we then have $\delta^*(G) = n - \max_{i, j \in [r]} n_{ij}$.
\end{construct}

Consider any graph $H$ on $h$ vertices with $\chi(H) = r \geq 3$. Since each row of $G= G((n_{ij}), r)$ induces an independent set in $G$, each copy $H'$ of $H$ in $G$ inherits an $r$-colouring from $G$ with colour classes $V(H') \cap V^j$ for $j \in [r]$. It follows from this that $H'$ has at least $\sigma(H)h$ vertices in each row $V^j$ of $G$, and that $|V(H') \cap V^j| - |V(H') \cap V^{j'}|$ is divisible by $\gcd(H)$ for any $j, j' \in [r]$.

Suppose first that $\gcd(H) > 1$, and fix any integer $n$. If $r$ divides $n$ then set $n_{11} = n/r + 1$, $n_{13} = n/r-1$ and $n_{ij} = n/r$ for each other pair $i, j \in [r]$, and note that we then have $|V^1| - |V^2| = 1$. Otherwise, set each $n_{ij}$ to be equal to either $\lfloor n/r \rfloor$ or $\lceil n/r \rceil$ in such a way that $\sum_{j \in [r]} n_{ij} = n$ for each $i \in [r]$ but $\sum_{i \in [r]} n_{i1} - \sum_{i \in [r]} n_{i2} = 1$; the latter implies that  $|V^1| - |V^2| = 1$. In either case we have $\delta^*(G) \geq n - \frac{n}{r} - 1 = (1 - \tfrac{1}{\chi^*(H)})n-1$ but $G$ has no perfect $H$-tiling. To see this, let $\T$ be an $H$-tiling in $G$. We observed above that $\gcd(H)$ divides $|V(H') \cap V^1| - |V(H') \cap V^2|$ for any $H' \in \T$. It follows that $\gcd(H)$ also divides $|V(\T) \cap V^1| - |V(\T) \cap V^2|$; since $|V^1| - |V^2| = 1$ and $\gcd(H) > 1$ this implies that $\T$ is not perfect. This shows that Theorem~\ref{theo:degalpha} is best-possible up to the $\alpha n$ error term for any $H$ with $\gcd(H) > 1$ and any $n$.

Now suppose instead that $\gcd(H) = 1$, and fix any integer $n$. For each $i \in [r]$ set $n_{i1} := \lceil \sigma(H) n \rceil-1$ and take $n_{i2}, \dots, n_{ir}$ to be as equal as possible with $\sum_{j = 1}^n n_j = n$. Then we have
\begin{align*}
\delta^*(G) &=n-\left\lceil \frac{n-\lceil \sigma(H) n\rceil+1}{r-1}\right\rceil \geq n - \frac{n-\sigma(H) n}{r-1} - 1
\\ & = \left(1 - \frac{1-\sigma(H)}{r-1} \right) n - 1
= \left(1-\frac{1}{\chi^*(H)}\right)n - 1.
\end{align*}
However we observed above that any copy of $H$ in $G$ has at least $\sigma h$ vertices in the row $V^1$, so any $H$-tiling in $G$ has size at most 
$$ \frac{|V^1|}{\sigma(H) h}  = \frac{r (\lceil \sigma(H) n \rceil-1)}{\sigma(H) h} < \frac{rn}{h},$$ 
so it is not perfect. This shows that Theorem~\ref{theo:degalpha} is best-possible up to the $\alpha n$ error term for any $H$ with $\gcd(H) = 1$ and any $n$.

\section{Concluding remarks}~

\medskip \noindent {\bf Comparison to non-partite results:} We note that Theorem~\ref{theo:degalpha} is strictly stronger than the analogous result in the non-partite setting. Indeed, let $H$ be a graph on $h$ vertices with $\chi(H) = r \geq 3$, and let $G$ be a balanced $r$-partite graph on $rn$ vertices with $\delta(G) \geq (1-1/\chi^*(H) + \alpha) n$, where $n$ is large and $h$ divides $rn$. We may arbitrarily delete at most $r$ copies of $H$ from $G$ so that the number of remaining vertices of $G$ is divisible by $r$, following which we partition the remaining vertices of $G$ into $r$ vertex classes of equal size uniformly at random. A standard probabilistic argument shows that with high probability we then have $\delta^*(G) \geq (1-1/\chi^*(H)) + \alpha n/2$, whereupon we may apply Theorem~\ref{theo:degalpha} to obtain a perfect $H$-tiling in $G$. 

On the other hand, the (non-partite) minimum degree of $G$ as in Theorem~\ref{theo:degalpha} may be as low as $(r-1)(1-1/\chi^*(H)) rn < (1-1/\chi^*(H)) rn$, which is too small for us to apply the analogous non-partite result. Similar comments apply to Theorem~\ref{theo:almosttiling}.

\medskip \noindent {\bf The case where $\chi(H) \neq r$:} In a similar manner, one can extend Theorem~\ref{theo:degalpha} to the case where $G$ has more vertex classes than $H$. Indeed, let $H$ be a graph on $h$ vertices with $\chi(H) = r \geq 3$, and let $G$ be a balanced $k$-partite graph on $kn$ vertices with $\delta(G) \geq (1-1/\chi^*(H) + \alpha) n$, where $n$ is large and divisible by $k$. If $k < r$ then $G$ does not contain even a single copy of $H$, whilst the case $k=r$ is dealt with by Theorem~\ref{theo:degalpha}. If instead $k > r$, then we first delete a small number of copies of $H$ in $G$ similarly as above, which allows us to assume that $n$ is divisible by $r$.  We then partition each vertex class $V_i$ of $G$ uniformly at random into $r$ parts $V_i^1, \dots, V_i^r$ each of size $n/r$. We then arrange these parts into $k$ vertex-disjoint balanced $r$-partite graphs $G_1, \dots, G_k$, where $V(G_\ell) = \bigcup_{j \in [r]} V_{\ell+j}^{j}$ (with addition taken modulo $k$). So each $G_\ell$ has $n$ vertices in total. Again a standard probabilistic argument shows that with high probability each $G_\ell$ has $\delta^*(G_\ell) \geq (1-1/\chi^*(H) + \alpha/2) \tfrac{n}{r}$. Theorem~\ref{theo:degalpha} then yields a perfect $H$-tiling in each $G_\ell$, and together these tilings form a perfect $H$-tiling in $G$.

\end{document}